\documentclass[12pt,leqno,draft]{article}
\usepackage{amsfonts}
\usepackage{amsmath, amsthm, amsfonts, amssymb, color}
\usepackage{mathrsfs}
\usepackage{color}
\usepackage{stmaryrd}
\makeatletter
\newcommand{\Spvek}[2][r]{%
  \gdef\@VORNE{1}
  \left(\hskip-\arraycolsep%
    \begin{array}{#1}\vekSp@lten{#2}\end{array}%
  \hskip-\arraycolsep\right)}

\def\vekSp@lten#1{\xvekSp@lten#1;vekL@stLine;}
\def\vekL@stLine{vekL@stLine}
\def\xvekSp@lten#1;{\def\temp{#1}%
  \ifx\temp\vekL@stLine
  \else
    \ifnum\@VORNE=1\gdef\@VORNE{0}
    \else\@arraycr\fi%
    #1%
    \expandafter\xvekSp@lten
  \fi}
\makeatother

\newtheorem{thm}{Theorem}[section]

\newtheorem{lem}[thm]{Lemma}

\newtheorem{rem}[thm]{Remark}
\theoremstyle{definition}

\newcommand{\scr}[1]{\mathscr #1}
\definecolor{wco}{rgb}{0.5,0.2,0.3}

\numberwithin{equation}{section} \theoremstyle{remark}

\newcommand{\ua}{\uparrow}

\title{{\bf    Log-Harnack Inequality and Exponential Ergodicity for Distribution Dependent CKLS and Vasicek Model}\footnote{Supported in
 part by  NNSFC (11801406).}
}
\author{
{\bf     Yifan Bai $^{b)}$, Xing Huang $^{a)}$}\\
\footnotesize{  a) Center for Applied Mathematics, Tianjin University, Tianjin 300072, China}\\
\footnotesize{  b) School of Mathematical Sciences, Peking University, Beijing 100091, China}\\
\footnotesize{  yifan.bai@stu.pku.edu.cn, xinghuang@tju.edu.cn}
}
\begin{document}
\allowdisplaybreaks
\def\R{\mathbb R}  \def\ff{\frac} \def\ss{\sqrt} \def\B{\mathbf
B}
\def\N{\mathbb N} \def\kk{\kappa} \def\m{{\bf m}}
\def\ee{\varepsilon}\def\ddd{D^*}
\def\dd{\delta} \def\DD{\Delta} \def\vv{\varepsilon} \def\rr{\rho}
\def\<{\langle} \def\>{\rangle} \def\GG{\Gamma} \def\gg{\gamma}
  \def\nn{\nabla} \def\pp{\partial} \def\E{\mathbb E}
\def\d{\text{\rm{d}}} \def\bb{\beta} \def\aa{\alpha} \def\D{\scr D}
  \def\si{\sigma} \def\ess{\text{\rm{ess}}}
\def\beg{\begin} \def\beq{\begin{equation}}  \def\F{\scr F}
\def\Ric{\text{\rm{Ric}}} \def\Hess{\text{\rm{Hess}}}
\def\e{\text{\rm{e}}} \def\ua{\underline a} \def\OO{\Omega}  \def\oo{\omega}
 \def\tt{\tilde} \def\Ric{\text{\rm{Ric}}}
\def\cut{\text{\rm{cut}}} \def\P{\mathbb P} \def\ifn{I_n(f^{\bigotimes n})}
\def\C{\scr C}   \def\G{\scr G}   \def\aaa{\mathbf{r}}     \def\r{r}
\def\gap{\text{\rm{gap}}} \def\prr{\pi_{{\bf m},\varrho}}  \def\r{\mathbf r}
\def\Z{\mathbb Z} \def\vrr{\varrho}
\def\L{\scr L}\def\Tt{\tt} \def\TT{\tt}\def\II{\mathbb I}
\def\i{{\rm in}}\def\Sect{{\rm Sect}}  \def\H{\mathbb H}
\def\M{\scr M}\def\Q{\mathbb Q} \def\texto{\text{o}} \def\LL{\Lambda}
\def\Rank{{\rm Rank}} \def\B{\scr B} \def\i{{\rm i}} \def\HR{\hat{\R}^d}
\def\to{\rightarrow}\def\l{\ell}\def\iint{\int}
\def\EE{\scr E}\def\no{\nonumber}
\def\A{\scr A}\def\V{\mathbb V}\def\osc{{\rm osc}}
\def\BB{\scr B}\def\Ent{{\rm Ent}}\def\3{\triangle}\def\H{\scr H}
\def\U{\scr U}\def\8{\infty}\def\1{\lesssim}\def\HH{\mathrm{H}}
 \def\T{\scr T}
 \def\W{\mathbb W}
\maketitle

\begin{abstract} In this paper, Wang's log-Harnack inequality and exponential ergodicity are derived for two types of distribution dependent SDEs: one is the CKLS model, where the diffusion coefficient is a power function of order $\theta$ with $\theta\in[\frac{1}{2},1)$; the other one is Vasicek model, where the diffusion coefficient only depends on distribution. Both models in the distribution independent case can be used to characterize the interest rate in mathematical finance.
\end{abstract} \noindent
 AMS subject Classification:\  60H10, 60H15.   \\
\noindent
 Keywords: Log-Harnack inequality; Exponential ergodicity; McKean-Vlasov SDEs; Wasserstein distance; Relative entropy. .
 \vskip 2cm

\section{Introduction}
The SDE
\begin{equation}\label{E0}
\d X_t=(\alpha-\delta X_t)\d t+|X_t|^\theta\d W_t,\ \ X_0\geq0,
\end{equation}
with $\alpha\geq0,\delta\geq 0,\theta\in[\frac{1}{2},1)$ is called CKLS model, which was introduced in \cite{CKLS}. It  can be used to characterize the evolution of the
interest rate in finance. By the Yamada-Watanabe approximation \cite{IW}, \eqref{E0} is strongly well-posed. In particular, when $\theta=\frac{1}{2}$, it is called Cox-Ingersoll-Ross (CIR) model \cite[Section 4.6]{C}. For CIR model, one can refer to \cite{CJM,CIR,WMC,YW,ZZ} for more introductions, applications, the convergence rate of various numerical methods and functional inequalities. Recently, \cite{HZ} has proved Wang's Harnack inequality and super Poincar\'{e} inequality for \eqref{E0} with $\theta\in(\frac{1}{2},1)$.

On the other hand, there are many results on the distribution dependent SDEs, also named McKean-Vlasov SDEs or mean field SDEs, in which the coefficients depend on the law of the solution, see for instance, \cite{BBP,CR,CF,HW,MV,RZ} and references therein.
 \cite{BH} investigated the strong well-posedness and propagation chaos of McKean-Vlasov SDEs with H\"{o}lder continuous diffusion coefficients, and the diffusion  is assumed to be distribution free.

 In this paper, we will first consider the distribution dependent version of \eqref{E0}, i.e. mean field CKLS model:
\begin{equation}\label{1.1}
\d X_t=(\alpha-\delta X_t)\d t+\gamma\E(X_t)\d t+|X_t|^{\theta} \d W_t,
\end{equation}
where $\frac{1}{2}\leq\theta<1$, ${\alpha}, {\delta}\geq0$, $\gamma\geq 0$ and $W_t$ is a one-dimensional Brownian motion on some complete filtration probability space $(\OO, \F, \{\F_{t}\}_{t\ge 0}, \P)$.
Noting that the diffusion in \eqref{1.1} is degenerate at $0$, we cannot directly use coupling by change of measure as in \cite{FYW1} to derive the log-Harnack inequality. Instead, we will adopt Girsanov's transform together with the method of coupling by change of measure to obtain the desired log-Harnack inequality. To this end, we will study the log-Harnack inequality for the decoupled SDEs. The crucial trick is to estimate $\E\int_0^t|X_s|^{-2\theta}\d s$, an upperbound of which will be provided in Lemma \ref{e-x} below by constructing appropriate test functions. Moreover, the exponential ergodicity in $L^1$-Wasserstein distance is also proved by the Yamada-Watanabe approximation in the case $\delta>\gamma$.

In addition, the Vasicek model
\begin{equation}\label{EVA}
\d X_t=(\gamma-\beta X_t)\d t+\sigma\d W_t
\end{equation}
with $\gamma,\beta,\sigma\in\R$ can also be used to characterize the interest rate and it was proposed in \cite{V}. Compared with \eqref{E0}, the solution to \eqref{EVA} can take negative values. Let $\scr P$ be the collection of all probability measures on $\R$ equipped with the weak topology. Consider the distribution dependent case of \eqref{EVA}:
\begin{align}\label{DD0}
\d X_t=(\gamma-\beta X_t)\d t+b(\L_{X_t})\d t+\sigma(\L_{X_t})\d W_t,
\end{align}
where $b,\sigma:\scr P\to \R$ are measurable.
Noting that the diffusion in \eqref{DD0} depends on distribution, which produces essential difficulty to study the log-Haranck inequality since the coupling by change of measure is unavailable. Fortunately, by observing the fact that the solution to \eqref{DD0} follows Gaussian distribution, we can estimate the relative entropy between two solutions from different initial distributions, which is equivalent to the log-Harnack inequality.

 The paper is organized as follows: In Section 2, we give results on the distribution dependent CKLS model \eqref{1.1}: the log-Harnack inequality and the exponential ergodicity in $L^1$-Wasserstein distance; The log-Harnack inequality as well as  exponential ergodicity in $L^2$-Wasserstein distance and in relative entropy for distribution dependent Vasicek model \eqref{DD0} will be given in Section 3.
\section{Distribution Dependent CKLS Model}
\subsection{Log-Harnack Inequality}
The monograph \cite{W} gives many applications of dimension-free Harnack inequality and a lot of models for it being true. For $p\in [1,\infty)$, let
$$\scr P_p := \left\{\mu\in \scr P\,: \,\, \mu(|\cdot|^p):=\int_{\R}|x|^p\,\mu(\d x)<\infty\right\}.$$
$\scr P^+_p $ is the subset of $\scr P_p$ with support on $[0,\infty)$.
It is well known that
$\scr P_p$ is a Polish space under the $L^p$-Wasserstein distance
$$\W_p(\mu_1,\mu_2):= \inf_{\pi\in \scr C(\mu_1,\mu_2)} \bigg(\int_{\mathbb{R}\times\mathbb{R}} |x-y|^p \,\pi(\d x,\d y)\bigg)^{1/p},\quad \mu_1,\mu_2\in \scr P_{p},$$ where $\scr C(\mu_1,\mu_2)$ is the set of all couplings
for $\mu_1$ and $\mu_2$.
In this section, we investigate the log-Harnack inequality for \eqref{1.1}. By Lemma \ref{L1} below, \eqref{1.1} with $\alpha,\delta,\gamma\geq 0$ and $X_0\geq 0$ is equivalent to
\begin{equation}\label{1.10}
\d X_t=(\alpha-\delta X_t)\d t+\gamma\E(X_t)\d t+X_t^{\theta} \d W_t.
\end{equation}
Noting that $|x^\theta-y^\theta|\leq |x-y|^\theta, x,y\geq 0$, \cite[Theorem 1.2]{BH} yields that \eqref{1.10} is well-posed. 
For any $\mu_0\in\scr P_1^+$, let $P_t^\ast \mu_0$ be the distribution of the solution to \eqref{1.10} with initial distribution $\mu_0$. Define
$$P_tf(\mu_0)=\int_{\R} f(x)(P_t^\ast\mu_0)(\d x), \ \ \mu_0\in\scr P_1^+, t\geq0, f\in\B_b([0,\infty)).$$
For any $\mu,\nu\in\scr P$, the relative entropy between $\mu,\nu$ is defined as
$$\mathrm{Ent}(\nu|\mu)=\left\{
  \begin{array}{ll}
    \nu(\log(\frac{\d \nu}{\d \mu})), & \hbox{$\nu\ll\mu$;} \\
    \infty, & \hbox{otherwise.}
  \end{array}
\right.$$
We shall introduce the intrinsic metric:
\begin{align*}\rho (x,y) = \int_{x \wedge y}^{x \vee y} \frac{\d r}{r^{\theta}}=\frac{(x \vee y)^{1-\theta}-(x \wedge y)^{1-\theta}}{1-\theta}=\sqrt{\frac{(x^{1-\theta}-y^{1-\theta})^2}{(1-\theta)^2}},\ \ x,y\in [0,\infty),
\end{align*}
and the $L^2$-Wasserstein distance induced by $\rho$: $$\mathbb{W}_{2,\rho}(\mu,\nu)=\inf_{\pi\in\C(\mu,\nu)}\left(\int_{[0,\infty)\times[0,\infty)}\rho(x,y)^2\pi(\d x,\d y)\right)^{\frac{1}{2}},\ \ \mu,\nu\in\scr P_1^+.$$
\begin{thm}\label{T-Har}  Assume $\delta>0$ and $\gamma\geq 0$. Then the following assertions hold.
\begin{enumerate}
\item[(1)] Assume  $\frac{1}{2}<\theta<1$
 and $\alpha \geq \frac{\theta}{2}$.
 For any $T>0$, $f\in \B^+_b([0,\infty))$ with $f>0$, $\mu_0,\nu_0\in \scr P^+_1$ with $\mu_0[(\cdot)^{1-2\theta}]<\infty$,
 the log-Harnack inequality holds, i.e.
  \beg{equation*}\beg{split}  P_T \log f(\nu_0)&\le  \log P_T f (\mu_0)
  +{\frac{2(1-\theta)(\delta -\frac{\theta}{2})  \mathbb{W}_{2,\rho}(\mu_0,\nu_0)^2 }{(\e^{2(1-\theta)(\delta -\frac{\theta}{2}) T}-1)}}\\
  &+\gamma^2(\e^{-2(\delta-\gamma)T}+1)\mathbb{W}_1(\mu_0,\nu_0)^2\Gamma(T,\delta, \alpha, \theta, \mu_0,\nu_0),
  \end{split}\end{equation*}
  where
  \begin{align*}&\Gamma(T,\delta, \alpha, \theta, \mu_0,\nu_0)\\
  &=\inf_{\varepsilon\in(0,\frac{\alpha}{3})}\Bigg\{\frac{ \frac{1}{2\theta-1}\mu_0[( \cdot)^{1-2\theta}]+ ((\delta^+)^{2\theta}\varepsilon^{1-2\theta}+ \varepsilon^{-\frac{1}{2\theta-1}})T}{\alpha-3\varepsilon}+\frac{ \varepsilon^{-1}{\frac{2(1-\theta)(\delta -\frac{\theta}{2})  \mathbb{W}_{2,\rho}(\mu_0,\nu_0)^2 }{(\e^{2(1-\theta)(\delta -\frac{\theta}{2}) T}-1)}}}{\alpha-3\varepsilon}\Bigg\}.
\end{align*}
\item[(2)] Assume  $\theta=\frac{1}{2}$
 and $\alpha > \frac{1}{2}$. Then for any $T>0$, $f\in \B^+_b([0,\infty))$ with $f>0$, $\mu_0,\nu_0\in\scr P_1^+$ satisfying 
  $\mu_0(|\log(\cdot)|)<\infty$, the log-Harnack inequality holds, i.e.
  \beg{equation*}\beg{split} P_T \log f(\nu_0)& \le  \log P_T f (\mu_0)
  +{\frac{(\delta -\frac{1}{4})  \mathbb{W}_{2,\rho}(\mu_0,\nu_0)^2  }{(\e^{(\delta -\frac{1}{4}) T}-1)}}\\
  &+\gamma^2(\e^{-2(\delta-\gamma)T}+1)\mathbb{W}_1(\mu_0,\nu_0)^2\bar{\Gamma}(T,\delta, \alpha, \mu_0,\nu_0),
  \end{split}\end{equation*}
 where \begin{align*}&\bar{\Gamma}(T,\delta, \alpha, \mu_0,\nu_0)=
\inf_{\varepsilon\in(0,\alpha-\frac{1}{2})} \frac{\mu_0(\log(\frac{\cdot+1}{\cdot})) +(\alpha+\delta^+)T+\varepsilon^{-1}{\frac{(\delta -\frac{1}{4})  \mathbb{W}_{2,\rho}(\mu_0,\nu_0)^2  }{(\e^{(\delta -\frac{1}{4}) T}-1)}}}{\alpha-\frac{1}{2}-\varepsilon}.
\end{align*}
 \end{enumerate}
 \end{thm}
\subsection{Proof of Theorem \ref{T-Har}}
Before giving the proof of Theorem \ref{T-Har}, we make some preparations. The first lemma tells us that the solution to \eqref{1.1} with non-negative initial value is non-negative.
\beg{lem}\label{L1} Assume $\alpha,\gamma\geq 0$. 
Let $X_t$ be the solution to \eqref{1.1} with $\F_0$-measurable non-negative initial value $X_0$. Then $\P$-a.s.
  \begin{equation*}
X_t\geq 0,\ \ t\geq0.
  \end{equation*}
  Moreover, it holds \begin{align}\label{Exp}
\E(X_t)=\e^{-(\delta-\gamma)t}\E(X_0)+\frac{\alpha}{\delta-\gamma}(1-\e^{-(\delta-\gamma)t}),\ \ t\geq 0,
\end{align}
here $\frac{\alpha}{\delta-\gamma}(1-\e^{-(\delta-\gamma)t})=\alpha t$ if $\delta=\gamma$.
 \end{lem}
 \begin{proof}
 For $\vv\in(0,1),$ noting that $\int_{\vv/\e}^\vv\ff{1}{x}\d
x=1$, there exists a continuous function
$\psi_{\vv}:[0,\infty)\to[0,\infty)$ with the support $[\vv/\e,\vv]$ such that
\begin{equation}\label{psi}
0\le\psi_{\vv}(x)\le \ff{2}{x},~~~x\in[\vv/\e,\vv],  ~~\int^\vv_{\vv/\e} \psi_{\vv}(r)\d r=1.
\end{equation}
Define
\begin{equation*}
\R\ni x\mapsto
V^0_{\vv}(x):=\int_0^{x^{-}}\int_0^y\psi_{\vv}(z)\d z\d y.
\end{equation*}
It is not difficult to see that
\begin{equation}\label{R10}
V^{0}_{\vv}(x)=0, x\geq -\vv/\e,\ \
x^{-}-\vv\le V^{0}_{\vv}(x)\le x^{-},~~x\in\R,
\end{equation}
\begin{align}\label{R00}
(V^0_{\vv})'(x)
\in[-1,0],\ \ x\leq  -\vv/\e,\ \ (V^0_{\vv})'(x)=0,\ \ x\geq  -\vv/\e,
\end{align}
and
\begin{equation}\label{R20}
 0\le (V_{\vv}^0)''(x)\le
\ff{2}{x^-}{\bf1}_{[\vv/\e,\vv]}(x^{-}),~~~~x\in\R.
\end{equation}
By It\^{o}'s formula, we get
\begin{align*}\d V^0_\varepsilon(X_t)&=(V^0_\varepsilon)'(X_t)(\alpha-\delta X_t+\gamma\E(X_t))\d t\\
&+(V^0_\varepsilon)'(X_t)|X_t|^\theta\d W_t+\frac{1}{2}(V^0_\varepsilon)''(X_t)|X_t|^{2\theta}\d t.
\end{align*}
For any $n\geq 1$, let $\tau_n=\inf\{t\geq 0: |X_t|\geq n\}$. We arrive at
\begin{align*}\E V^0_\varepsilon(X_{s\wedge \tau_n})&=\E V^0_\varepsilon(X_0)+\E\int_0^{s\wedge\tau_n}(V^0_\varepsilon)'(X_t)(\alpha-\delta X_t+\gamma\E(X_t))\d t\\
&\qquad\qquad\qquad+\frac{1}{2}\E\int_0^{s\wedge\tau_n}(V^0_\varepsilon)''(X_t)|X_t|^{2\theta}\d t.
\end{align*}
 Letting $n\to\infty$, the dominated convergence theorem yields
 \begin{align}\label{PhP}\nonumber\E V^0_\varepsilon(X_{s})&=\E V^0_\varepsilon(X_0)+\E\int_0^{s}(V^0_\varepsilon)'(X_t)(\alpha-\delta X_t+\gamma\E(X_t))\d t\\
 &\qquad\qquad\qquad+\frac{1}{2}\E\int_0^{s}(V^0_\varepsilon)''(X_t)|X_t|^{2\theta}\d t.\end{align}
This together with \eqref{R10}-\eqref{R20}, $\gamma\geq 0, X_0\geq 0$, $\varepsilon\in(0,1)$ implies
\begin{align} \label{V-0}\nonumber \E V^0_\varepsilon(X_s)&\leq\E V^0_\varepsilon(X_0)+\E\int_0^s(V^0_\varepsilon)'(X_t)(\alpha-\delta X_t^++\gamma\E(X_t^+))\d t\\
&+\int_0^s(|\delta| +\gamma)\E(X_t^-)\d t+\int_0^s1_{[\varepsilon/\e,\varepsilon]}(X_t^-)\d t\\
\nonumber &\leq \int_0^s(|\delta| +\gamma)\E(X_t^-)\d t+\int_0^s1_{[\varepsilon/\e,\varepsilon]}(X_t^-)\d t.
\end{align}
Letting $\varepsilon\to 0$, the dominated convergence theorem, \eqref{R10}, \eqref{V-0} and Gronwall's inequality yield
\begin{align*}
\E(X_t^-)=0, t\geq 0.
\end{align*}
This combined with the continuity of $X_t$ in $t$ implies that $\P$-a.s.
  \begin{equation*}
X_t\geq 0,\ \ t\geq0.
  \end{equation*} Finally,
by the same argument to obtain \eqref{PhP}, we have
\begin{align*}
\E(X_t)= \E(X_0)+\alpha t-\int_0^t(\delta-\gamma)\E(X_s)\d s.
\end{align*}
This implies \eqref{Exp} immediately. So, we complete the proof.
 \end{proof}
With the above preparations in hand, we are in the position to complete the proof of Theorem \ref{T-Har}.
\begin{proof}[Proof of Theorem \ref{T-Har}]
Let $\mu_t=P_t^\ast \mu_0,\nu_t=P_t^\ast\nu_0$. 
We divide the proof into three steps.

{\bf Step (I).} For any $x>0$, consider
\begin{align}\label{DXM}\d X^{x,\mu}_t=(\alpha-\delta X^{x,\mu}_t)\d t+\gamma\mu_t(\cdot)\d t+(X^{x,\mu}_t)^{\theta} \d W_t, \ \ X_0^{x,\mu}=x.
\end{align}
For simplicity, we denote $X_t=X^{x,\mu}_t$.
For any $m\geq 1$, define
\begin{align}\label{BET}\beta_m=\inf\left\{t\geq 0: X_t\leq \frac{1}{m}\right\}.\end{align}
Then by Lemma \ref{e-x} below with $\alpha_t=\alpha+\gamma\mu_t$ and $\zeta(t)=0$, $\P$-a.s. $\lim_{m\to\infty}\beta_m=\infty$ and so $\P$-a.s. $X_t>0, t\geq 0$.
Letting \begin{align}\label{alt}\alpha^\nu_t=\alpha+\gamma\nu_t(\cdot),\ \ t\geq 0,\end{align}
we have $\alpha^\nu_t\geq \alpha$ due to $\gamma\nu_t(\cdot)\geq 0$. We rewrite \eqref{DXM} as
\begin{align}\label{TIW}\d X_t=(\alpha^\nu_t-\delta X_t)\d t+X_t^{\theta} \d \tilde{W}_t,\ \ X_0=x,\end{align}
here
$$\tilde{W}_t= W_t+\int_0^tX_s^{-\theta}(\gamma\mu_s(\cdot)-\gamma\nu_s(\cdot))\d s.$$
Let $$R_s=\exp\left\{-\int_0^sX_t^{-\theta}(\gamma\mu_t(\cdot)-\gamma\nu_t(\cdot))\d W_t-\frac{1}{2}\int_0^s|X_t^{-\theta}(\gamma\mu_t(\cdot)-\gamma\nu_t(\cdot))|^2\d t\right\}, \ \ s\in[0,T],$$
\eqref{Exp} implies that for any $m\geq 1$, $(R_{s\wedge\beta_m})_{s\in[0,T]}$ is a martingale and Girsanov's theorem yields that $(\tilde{W}_{s\wedge \beta_m})_{s\in[0,T]}$ is a one-dimensional Brownian motion under $\Q_T^m=R_{T\wedge \beta_m}\P$. Moreover, it follows from \eqref{Exp} that
\begin{align}\label{Rlo}\nonumber&\E(R_{s\wedge\beta_m}\log R_{s\wedge\beta_m})\\
&\leq\frac{1}{2}\E^{\Q^m_T}\int_0^{s\wedge \beta_m}X_t^{-2\theta}|\gamma\mu_t(\cdot)-\gamma\nu_t(\cdot)|^2\d t\\
\nonumber&\leq \frac{1}{2}\gamma^2(\E|X_0-Y_0|)^2(\e^{-2(\delta-\gamma)s}+1)\E^{\Q^m_T}\int_0^{s\wedge \beta_m}X_t^{-2\theta}\d t,\ \ s\in[0,T].
\end{align}
By Lemma \ref{e-x} below for $W_t=\tilde{W}_{t\wedge \beta_m}$, $\zeta=0$ and $\P=\Q_T^{m}$, we have
$$\sup_{m\geq 1}\E^{\Q^m_T}\int_0^{T}X_t^{-2\theta}\d t<\infty,$$
which yields
$$\sup_{m\geq 1}\E(R_{s\wedge\beta_m}\log R_{s\wedge\beta_m})<\infty,\ \ s\in[0,T].$$ Then it follows from the martingale convergence theorem and the fact $\P$-a.s. $\lim_{m\to\infty}\beta_m=\infty$ that $\E R_s=1,s\in[0,T]$, which means that $\{R_s\}_{s\in[0,T]}$ is a martingale.

{\bf Step (II).} By Step (I), we know that $(\tilde{W}_t)_{t\in[0,T]}$ is a one-dimensional Brownian motion under the probability measure $\Q_T=R_T\P$. Let $X_t^{y,\nu}$ solve \eqref{DXM} with $(y,\nu)$ replacing $(x,\mu)$ for $y\geq0$. Let $Y_t$ solve
  $$\d Y_t=(\alpha^\nu_t-\delta Y_t)\d t+Y_t^{\theta} \d \tilde{W}_t+Y_t^{\theta}1_{[0,\tau)}(t)\xi_t\d t ,\ \ Y_0=y,$$
  where
  $$\xi_t=\frac{2(\delta-\frac{\theta}{2})(x^{1-\theta}-y^{1-\theta})\e^{(1-\theta)(\delta-\frac{\theta}{2})t}} {\e^{2(1-\theta)(\delta-\frac{\theta}{2})T}-1},\ \ t\geq 0$$
  and $\tau=\inf\{t\geq 0: X_t=Y_t\}$. Set $Y_t=X_t, t\geq \tau$.
According to the proof of \cite[Theorem 2.1(1)]{HZ}, $\Q_T(\tau\leq T)=1$ and $\{\bar{W}_t\}_{t\in[0,T]}$ with $\bar{W}_t=\tilde{W}_t+\int_0^t\xi_s1_{[0,\tau)}(s)\d s$ is a one-dimensional Brownian motion under $\bar{\Q}_T=\bar{R}_\tau\Q_T$, where
$$\bar{R}_t=\exp\left\{-\int_0^{t\wedge\tau}\xi_s\d \tilde{W}_s-\frac{1}{2}\int_0^{t\wedge\tau}|\xi_s|^2\d s\right\},\ \ t\in[0,T].$$
 Moreover, we have $\L_{Y_t}|\bar{\Q}_T=\L_{X_t^{y,\nu}}, t\in[0,T]$, $\bar{\Q}_T$-a.s. $X_T=Y_T$ and
  \beg{align}\label{BAR} \E^{\bar{\Q}_T}\log (\bar{R}_\tau)=\frac{1}{2}\int_0^T|\xi_s|^2\d s
  \leq{\frac{(1-\theta)(\delta -\frac{\theta}{2})  \rho(x,y)^2 }{(\e^{2(1-\theta)(\delta -\frac{\theta}{2}) T}-1)}}.
  \end{align}
{\bf Step (III).}
Noting that
\begin{align*}
\d W_t=\d \tilde{W}_t-X_t^{-\theta}(\gamma\mu_t(\cdot)-\gamma\nu_t(\cdot))\d t=\d \bar{W}_t-\xi_t1_{[0,\tau)}(t)\d t-X_t^{-\theta}(\gamma\mu_t(\cdot)-\gamma\nu_t(\cdot))\d t,
\end{align*}
we conclude that
\begin{align}\label{RTQ}
\nonumber \E^{\bar{\Q}_T}\log R_T&=\E^{\bar{\Q}_T}\int_0^TX_t^{-\theta}(\gamma\mu_t(\cdot)-\gamma\nu_t(\cdot))\xi_t1_{[0,\tau)}(t)\d t\\
&+\frac{1}{2}\E^{\bar{\Q}_T}\int_0^TX_t^{-2\theta}|\gamma\mu_t(\cdot)-\gamma\nu_t(\cdot)|^2\d t\\
\nonumber&\leq \E^{\bar{\Q}_T}\int_0^TX_t^{-2\theta}|\gamma\mu_t(\cdot)-\gamma\nu_t(\cdot)|^2\d t+\frac{1}{2}\E^{\bar{\Q}_T}\int_0^T|\xi_t|^2\d t.
\end{align}
So, \eqref{BAR} and \eqref{RTQ} yield \begin{align}\label{loR}\E^{\bar{\Q}_T}\log (\bar{R}_\tau R_T)\leq {\frac{2(1-\theta)(\delta -\frac{\theta}{2})   \rho(x,y)^2 }{(\e^{2(1-\theta)(\delta -\frac{\theta}{2}) T}-1)}}+\E^{\bar{\Q}_T}\int_0^TX_t^{-2\theta}|\gamma\mu_t(\cdot)-\gamma\nu_t(\cdot)|^2\d t.
\end{align}
Applying Young's inequality,  for any $f\in \B^+_b([0,\infty))$ with $f>0$, we have
  \beg{equation}\label{logxy}\E\log f(X_T^{y,\nu})=\E^{\bar{\Q}_T}\log f(Y_T)=\E^{\bar{\Q}_T}\log f(X_T)\le  \log \E f (X_T^{x,\mu})+\E^{\bar{\Q}_T}\log (\bar{R}_\tau R_T).
\end{equation}
Rewrite \eqref{TIW} as
\begin{align*}\d X_t=(\alpha^\nu_t-\delta X_t)\d t-X_t^{\theta}\xi(t)1_{[0,\tau)} (t)\d t+X_t^{\theta} \d \bar W_t,\ \ X_0=x.\end{align*}
Applying Lemma \ref{e-x} below for $W_t=\bar{W}_t$, $\zeta(t)=\xi_t1_{[0,\tau)} (t)$ and $\P=\bar{\Q}_T$, combining \eqref{Exp} and \eqref{loR}-\eqref{logxy}, when $\frac{1}{2}<\theta<1$
 and $\alpha \geq \frac{\theta}{2}$,
  \beg{equation}\beg{split}\label{lo1}  \E\log f(X_T^{y,\nu})&\le  \log \E f (X_T^{x,\mu})
  +{\frac{2(1-\theta)(\delta -\frac{\theta}{2})  \rho(x,y)^2 }{(\e^{2(1-\theta)(\delta -\frac{\theta}{2}) T}-1)}}\\
  &+\gamma^2(\e^{-2(\delta-\gamma)T}+1)\W_1(\mu_0,\nu_0)^2\Gamma(T,\delta, \alpha, \theta, \delta_x,\delta_y),
  \end{split}\end{equation}
and when $\theta=\frac{1}{2}$
 and $\alpha > \frac{1}{2}$,
  \beg{equation}\beg{split}\label{lo2} \E\log f(X_T^{y,\nu})&\le  \log \E f (X_T^{x,\mu})
  +{\frac{(\delta -\frac{1}{4}) \rho(x,y)^2  }{(\e^{(\delta -\frac{1}{4}) T}-1)}}\\
  &+\gamma^2(\e^{-2(\delta-\gamma)T}+1)\W_1(\mu_0,\nu_0)^2\bar{\Gamma}(T,\delta, \alpha, \delta_x,\delta_y).
  \end{split}\end{equation}
Noting that both $\mu_0[(\cdot)^{1-2\theta}]<\infty$ and $\mu_0(|\log(\cdot)|)<\infty$ yield $$P_T^\ast\mu_0=\int_{(0,\infty)}\L_{X_T^{x,\mu}}\mu_0(\d x),\ \ P_T^\ast\nu_0=\int_{[0,\infty)}\L_{X_T^{y,\nu}}\nu_0(\d x),$$
for any $\pi\in\C(\mu_0,\nu_0)$, taking expectation with respect to $\pi$, using Jensen's inequality, and then taking infimum in $\pi$ on the two sides of \eqref{lo1} and \eqref{lo2}, we complete the proof.
\end{proof}
Let $\alpha_t$ be a measurable function from $[0,\infty)$ to $[\alpha,\infty)$, $\zeta(t)$ be a progressively measurable process with $\E\int_0^t|\zeta(s)|^2\d s $ locally bounded in $t$ and $X_t^{\zeta}$ be a non-negative solution to the SDE
\begin{align}\label{XQ0}\d X_t=(\alpha_t-\delta X_t)\d t-X_t^{\theta}\zeta(t)\d t+X_t^{\theta} \d W_t,\ \ X_0=x> 0.\end{align}
For any $m\geq 1$, let $\beta^\zeta_m$ be defined in \eqref{BET} with $X_t^\zeta$ replacing $X_t$.
\begin{lem}\label{e-x} The following assertions hold.
\begin{enumerate}
\item[(1)] Assume
$\theta\in(\frac{1}{2}, 1), \alpha>0$, we have $\lim_{m\to\infty}\beta_m^\zeta=\infty$ and \begin{align*}\E\int_0^T(X_t^\zeta)^{-2\theta}\d t&\leq \inf_{\varepsilon_1\in(0,\frac{\alpha}{3})}\frac{ \frac{1}{2\theta-1} x^{1-2\theta}+ (\delta^+)^{2\theta}\varepsilon_1^{1-2\theta}T+ \varepsilon_1^{-\frac{1}{2\theta-1}}T+\varepsilon_1^{-1}\E\int_0^T\zeta(t)^2\d t}{\alpha-3\varepsilon_1}.
\end{align*}
\item[(2)] Assume
 $\theta=\frac{1}{2}$ and $\alpha>\frac{1}{2}$, we obtain $\lim_{m\to\infty}\beta_m^\zeta=\infty$ and
\begin{align*}\E\int_0^T(X_t^\zeta)^{-1}\d t&\leq \inf_{\varepsilon_1\in(0,\alpha-\frac{1}{2})} \frac{\log(\frac{x+1}{x})+(\alpha+\delta^+)T+\varepsilon_1^{-1}\E\int_0^T\zeta(t)^2\d t}{\alpha-\frac{1}{2}-\varepsilon_1}.
\end{align*}
\end{enumerate}
\end{lem}
\begin{proof} For simplicity, we omit the superscript, i.e. we denote $X_t=X_t^\zeta$ and $\beta_m=\beta^\zeta_m$.

(1) Define $$V(x)=\frac{1}{2\theta-1}x^{1-2\theta},\ \ x> 0.$$
Then it is clear that
\begin{align}\label{VIN}V(x)> 0,\ \ V'(x)=-x^{-2\theta},\ \ V''(x)=2\theta x^{-2\theta-1},\ \ x>0,\ \ \lim_{x\to0}V(x)=\infty.
\end{align}
It follows from It\^{o}'s formula, \eqref{VIN} and $\alpha_t\geq \alpha$ that
\begin{equation*}\begin{split}
\d V(X_t)&\leq (\alpha-\delta X_t-X_t^\theta\zeta(t))(-X_t^{-2\theta})\d t+X_t^{\theta}(-X_t^{-2\theta}) \d W_t+\theta X_t^{-2\theta-1}X_t^{2\theta}\d t,\ \ t\leq\beta_m.
\end{split}\end{equation*}
So, we have
\begin{align}\label{V-V}
&\nonumber V(X_{s\wedge\beta_m})-V(x)\\
&\leq \int_0^{s\wedge\beta_m}(-\alpha X_t^{-2\theta}+\delta^+ X_t^{-2\theta+1}+\theta  X_t^{-1}+X_t^{-\theta}|\zeta(t)|)\d t-\int_0^{s\wedge\beta_m}X_t^{-\theta} \d W_t.
\end{align}
Noting that $-2\leq -2\theta<-1$, $\alpha>0$, Young's inequality implies that for any $\varepsilon_1\in(0,\frac{\alpha}{3})$,

\begin{align}\label{ttt}\nonumber&\delta^+ X_t^{-2\theta+1}=((\delta^+)^{2\theta}\varepsilon_1^{1-2\theta})^{\frac{1}{2\theta}} (\varepsilon_1X_t^{-2\theta})^{\frac{2\theta-1}{2\theta}}\leq (\delta^+)^{2\theta}\varepsilon_1^{1-2\theta}+ \varepsilon_1X_t^{-2\theta},\\
&\theta X_t^{-1}= (\theta^{\frac{2\theta}{2\theta-1}}\varepsilon_1^{-\frac{1}{2\theta-1}})^{\frac{2\theta-1}{2\theta}} (\varepsilon _1X_t^{-2\theta})^{\frac{1}{2\theta}}
\leq \varepsilon_1^{-\frac{1}{2\theta-1}}+ \varepsilon _1X_t^{-2\theta},\\
\nonumber&|\zeta(t)| X_t^{-\theta}= (|\zeta(t)|^2\varepsilon_1^{-1})^{\frac{1}{2}} (\varepsilon _1X_t^{-2\theta})^{\frac{1}{2}}
\leq |\zeta(t)|^2\varepsilon_1^{-1}+ \varepsilon _1X_t^{-2\theta}.
\end{align}
Combining \eqref{V-V}-\eqref{ttt}, we conclude that for any $\varepsilon_1\in(0,\frac{\alpha}{3})$, it holds
\begin{align*}
\E V(X_{s\wedge\beta_m})\leq V(x)+(\delta^+)^{2\theta}\varepsilon_1^{1-2\theta}s+ \varepsilon_1^{-\frac{1}{2\theta-1}}s+\E\int_0^s|\zeta(t)|^2\varepsilon_1^{-1}\d t,\ \ s\geq 0.
\end{align*}
This implies that
\begin{align*}\P(\beta_m\leq s)&\leq (2\theta-1)m^{1-2\theta}\E [V(X_{s\wedge\beta_m}) 1_{\{\beta_m\leq s\}}]\\
&\leq (2\theta-1)m^{1-2\theta}\left( V(x)+(\delta^+)^{2\theta}\varepsilon_1^{1-2\theta}s+ \varepsilon_1^{-\frac{1}{2\theta-1}}s+\E\int_0^s|\zeta(t)|^2\varepsilon_1^{-1}\d t\right), \ \ s\geq 0.
\end{align*}
So, $\P$-a.s. $\lim_{m\to\infty}\beta_m=\infty$ and thus $\P$-a.s. $X_t>0,t\geq 0$.
Moreover, substituting \eqref{ttt} into \eqref{V-V} and taking expectation, we get
\begin{align*}\E\int_0^{T\wedge\beta_m}X_t^{-2\theta}\d t&\leq \inf_{\varepsilon_1\in(0,\frac{\alpha}{3})}\frac{ \frac{1}{2\theta-1} x^{1-2\theta}+ (\delta^+)^{2\theta}\varepsilon_1^{1-2\theta}T+ \varepsilon_1^{-\frac{1}{2\theta-1}}T+\varepsilon_1^{-1}\E\int_0^T\zeta(t)^2\d t}{\alpha-3\varepsilon_1}.
\end{align*}
Letting $m\to\infty$, Fatou's lemma derives (1).

(2) Define $$\bar{V}(x)=\log(x+1)-\log x, \ \ x>0.$$
Then we have
$$\bar{V}(x)> 0, \ \ \bar{V}'(x)=(x+1)^{-1}-x^{-1},\ \ \bar{V}''(x)=x^{-2}-(x+1)^{-2},\ \ x>0,\ \ \lim_{x\to0}\bar{V}(x)=\infty.$$
 By It\^{o}'s formula, we arrive at
\begin{align}\label{VBA}
\nonumber &\bar{V}(X_{s\wedge\beta_m})-\bar{V}(x)\\
&\leq \int_0^{s\wedge\beta_m}(\alpha-\delta X_t-X_t^{\frac{1}{2}}\zeta(t))[(X_t+1)^{-1}-X_t^{-1}]\d t\\
\nonumber &+\frac{1}{2}\int_0^{s\wedge\beta_m}X_t[X_t^{-2}-(X_t+1)^{-2}]\d t+\int_0^{s\wedge\beta_m}X_t^{\frac{1}{2}}[(X_t+1)^{-1}-X_t^{-1}] \d W_t\\
\nonumber&\leq \int_0^{s\wedge\beta_m}\left(-\alpha+\frac{1}{2}\right) X_t^{-1}\d t+(\alpha+\delta^+)s\\
\nonumber&+\int_0^{s\wedge\beta_m}X_t^{-\frac{1}{2}}|\zeta(t)|\d t+\int_0^{s\wedge\beta_m}X_t^{\frac{1}{2}}[(X_t+1)^{-1}-X_t^{-1}] \d W_t.
\end{align}
Since $\alpha>\frac{1}{2}$ and
$$|\zeta(t)| X_t^{-\frac{1}{2}}= (|\zeta(t)|^2\varepsilon_1^{-1})^{\frac{1}{2}} (\varepsilon _1X_t^{-1})^{\frac{1}{2}}
\leq |\zeta(t)|^2\varepsilon_1^{-1}+ \varepsilon _1X_t^{-1},$$
for any $\varepsilon_1\in(0,\alpha-\frac{1}{2})$, \eqref{VBA} implies
\begin{align*}
&\E \bar{V}(X_{s\wedge\beta_m})\leq \bar{V}(x)+(\alpha+\delta^+)s+\E\int_0^s|\zeta(t)|^2\varepsilon_1^{-1}\d t.
\end{align*}
As a result, it holds
\begin{align*}\P(\beta_m\leq s)&\leq \left[\log(\frac{1}{m}+1)-\log\frac{ 1}{m}\right]^{-1}\E [\bar{V}(X_{s\wedge\beta_m}) 1_{\{\beta_m\leq s\}}]\\
&\leq \left[\log(\frac{1}{m}+1)-\log\frac{ 1}{m}\right]^{-1}\left(\bar{V}(x)+(\alpha+\delta^+)s+\E\int_0^s|\zeta(t)|^2\varepsilon_1^{-1}\d t\right), \ \ s\geq 0,
\end{align*}
which implies that $\P$-a.s. $\lim_{m\to\infty}\beta_m=\infty$ and thus $\P$-a.s. $X_t>0,t\geq 0$. Finally, it follows from \eqref{VBA} that
\begin{align*}\E\int_0^{T\wedge\beta_m}X_t^{-1}\d t\leq \inf_{\varepsilon_1\in(0,\alpha-\frac{1}{2})} \frac{\log(\frac{x+1}{x})+(\alpha+\delta^+)T+\varepsilon_1^{-1}\E\int_0^T\zeta(t)^2\d t}{\alpha-\frac{1}{2}-\varepsilon_1}.
\end{align*}
Letting $m\to\infty$, Fatou's lemma completes the proof.
\end{proof}
\subsection{Exponential Ergodicity in Wasserstein Distance}
Recall that $P_t^\ast\mu$ is the distribution of the solution to \eqref{1.10} with initial distribution $\mu\in \scr P_1^+$.
\begin{thm}\label{TEX} Assume that $\alpha\geq 0$, $\delta>\gamma\geq 0$. Then $P_t^*$  has a unique  invariant probability measure $\mu\in \scr P_1$ satisfying
$$\W_1(P_t^\ast\nu, \mu)\leq \e^{-(\delta-\gamma)t}\W_1(\nu, \mu),\ \ \nu\in\scr P_1^+.$$
\end{thm}
\begin{proof} Let
$\psi_{\vv}$ be defined in \eqref{psi}.
Define
\begin{equation*}
\R\ni x\mapsto
V_{\vv}(x):=\int_0^{|x|}\int_0^y\psi_{\vv}(z)\d z\d y
\end{equation*}
It is not difficult to see that
\begin{equation}\label{R1}
|x|-\vv\le V_{\vv}(x)\le |x|,~~ \mbox{sgn}(x) V_{\vv}'(x)
\in[0,1],~~x\in\R,
\end{equation}
and
\begin{equation}\label{R2}
 0\le V_{\vv}''(x)\le
\ff{2}{|x|}{\bf1}_{[\vv/\e,\vv]}(|x|),~~~~x\in\R.
\end{equation}
Let $X_t$ and $Y_t$ be solutions to \eqref{1.10} with non-negative initial values $X_0$ and $Y_0$ respectively.
For any $\varepsilon>0$, it follows from It\^{o}'s formula that
\begin{align*}\d V_\varepsilon(X_t-Y_t)&=V_\varepsilon'(X_t-Y_t)(-\delta(X_t-Y_t)+\gamma(\E(X_t)-\E(Y_t)))\d t\\
&+V_\varepsilon'(X_t-Y_t)[X_t^\theta-Y_t^\theta]\d W_t\\
&+\frac{1}{2}V_\varepsilon''(X_t-Y_t)[X_t^\theta-Y_t^\theta]^2\d t.
\end{align*}
By \eqref{R2} and the inequality $|x^\theta-y^\theta|\leq |x-y|^\theta, x,y\geq 0$, we have
$$\frac{1}{2}V_\varepsilon''(X_t-Y_t)[X_t^\theta-Y_t^\theta]^2\leq
\varepsilon^{2\theta-1}1_{[\varepsilon/\e,\varepsilon]}(|X_t-Y_t|).$$
So, by the same argument to obtain \eqref{PhP}, \eqref{R1} yields that
\begin{align*}\E V_\varepsilon(X_s-Y_s)
&\leq \E V_\varepsilon(X_0-Y_0)+\int_0^s-(\delta-\gamma)\E |X_t-Y_t|\d t\\
&+\int_0^s\varepsilon^{2\theta-1}1_{[\varepsilon/\e,\varepsilon]}(|X_t-Y_t|)\d t.
\end{align*}
Letting $\varepsilon\to0$ and using \eqref{R1}, we arrive at
\begin{align*}\E |X_s-Y_s|
&\leq \E |X_0-Y_0|+\int_0^s-(\delta-\gamma)\E |X_t-Y_t|\d t.
\end{align*}
Gronwall's inequality implies that
\begin{align*}\E |X_s-Y_s|\leq \e^{-(\delta-\gamma)s}\E |X_0-Y_0|.
\end{align*}
Since $\delta>\gamma$, it is standard to prove that $P_t^\ast$ has a unique invariant probability $\mu$ with support on $[0,\infty)$ and satisfying
$$\mathbb{W}_1(P_t^\ast\nu,\mu)\leq \e^{-(\delta-\gamma)t}\mathbb{W}_1(\nu,\mu),\ \ \nu\in\scr P_1^+,$$
see \cite[Proof of Theorem 3.1(2)]{FYW1}.
\end{proof}
\section{Distribution Dependent Vasicek Model}
In this section, we consider the distribution dependent Vasicek model \eqref{DD0}.
Assume that
\begin{enumerate}
\item[\bf (H1)] There exist constants $L_b,L_{\sigma}\geq 0$ such that
\begin{align*}
&|b(\mu)-b(\nu)|\leq L_b\mathbb{W}_{2}(\mu,\nu),\ \ |\sigma(\mu)-\sigma(\nu)|\leq L_{\sigma}\mathbb{W}_{2}(\mu,\nu),\ \ \mu,\nu\in\scr P_{2}.
\end{align*}
\item[\bf (H2)] There exists a constant $K\geq 1$ such that
\begin{align*} &K^{-1}\leq \sigma^2(\mu)\leq K,\ \ \mu\in\scr P_{2}.
\end{align*}
\end{enumerate}
Under {\bf(H1)}, \eqref{DD0} is strongly well-posed according to \cite{W}. For any $\mu_0\in\scr P_2$, let $P_t^\ast \mu_0$ be the distribution of the solution to \eqref{DD0} with initial distribution $\mu_0$, and define
$$P_tf(\mu_0)=\int_{\R} f(x)(P_t^\ast\mu_0)(\d x), \ \ \mu_0\in\scr P_2, t\geq0, f\in\B_b(\R).$$
It is standard from {\bf(H1)} that
\begin{align}\label{WET}\mathbb{W}_2(P_t^\ast \mu_0, P_t^\ast \nu_0)\leq \e^{(-\beta+L_b+\frac{L_\sigma^2}{2})t}\mathbb{W}_2(\mu_0, \nu_0),\ \ t\geq 0.
\end{align}

\begin{thm}\label{log} The log-Harnack inequality holds, i.e.
\beg{equation*}
P_t\log f(\mu_0)\le  \log P_t f (\nu_0)
  +\Sigma(t)\mathbb{W}_2(\mu_0,\nu_0)^2, \ \ f\in \B_b(\mathbb{R}), f>0,t>0,\mu_0,\nu_0\in\scr P_2
\end{equation*}
with
\begin{align*}
\Sigma(t)
&=\frac{2\beta K}{\e^{2\beta t}-1}+\frac{2\beta K}{\e^{2\beta t}-1}L^2_b\frac{(\e^{(L_b+\frac{L_\sigma^2}{2})t}-1)^2}{(L_b+\frac{L_\sigma^2}{2})^2}\\
&\qquad+\frac{K+1}{2}\left(\frac{1-\e^{-2\beta t}}{2\beta}\right)^{-2}K^3L_\sigma^2\e^{-4\beta t}\frac{(\e^{(\beta+L_b+\frac{L_\sigma^2}{2})t}-1)^2}{(\beta+L_b+\frac{L_\sigma^2}{2})^2},
\end{align*}
here $\frac{\e^{\delta t}-1}{\delta}=t$ when $\delta=0$.
\end{thm}
\begin{proof}
For any $x\in\R$, let
\begin{align*}&\Gamma^{\mu_0,x}_t=\e^{-\beta t}x+\int_0^t\e^{-\beta(t-s)}[\gamma+b(P_s^\ast \mu_0)]\d s,\ \ \Sigma^{\mu_0}_t=\int_0^t|\e^{-\beta(t-s)}\sigma(P_s^\ast\mu_0)|^2\d s,\ \ t\geq 0
\end{align*}
and define
\begin{align*}
X^{\mu_0,x}_t=\Gamma^{\mu_0,x}_t+\int_0^t\e^{-\beta(t-s)}\sigma(P_s^\ast \mu_0)\d W_s,\ \ t\geq 0.
\end{align*}
Then it is clear that
\begin{align}\label{REP}P_t^\ast\mu_0=\int_{\R}\L_{X^{\mu_0,x}_t}\mu_0(\d x),\ \ t\geq 0,\end{align}
and
\begin{align}\label{DDV000}
 \frac{\d \L_{X^{\mu_0,x}_t}}{\d z}(z)=\frac{1}{\sqrt{2\pi\Sigma^{\mu_0}_t}}\exp\left\{-\frac{(z-\Gamma^{\mu_0,x}_t)^2}{2\Sigma^{\mu_0}_t}\right\},\ \ t>0.
\end{align}
By {\bf (H2)}, we have
\begin{align}\label{EQ1}
&\frac{1-\e^{-2\beta t}}{2\beta}K^{-1}\leq\Sigma^{\mu_0}_t\leq \frac{1-\e^{-2\beta t}}{2\beta}K, \ \ t\geq 0.
\end{align}
Moreover, {\bf (H1)-(H2)} and \eqref{WET} imply
\begin{align}\label{EQ2}
\nonumber |\Sigma^{\mu_0}_t-\Sigma^{\nu_0}_t|&\leq 2\sqrt{K}L_\sigma\int_0^t\e^{-2\beta(t-s)}\mathbb{W}_{2}(P_s^\ast \mu_0,P_s^\ast\nu_0)\d s,\\
&\leq 2\sqrt{K}L_\sigma\mathbb{W}_2(\mu_0, \nu_0)\int_0^t\e^{-2\beta(t-s)}\e^{(-\beta+L_b+\frac{L_\sigma^2}{2})s}\d s\\
\nonumber &\leq 2\sqrt{K}L_\sigma\mathbb{W}_2(\mu_0, \nu_0)\e^{-2\beta t}\frac{\e^{(\beta+L_b+\frac{L_\sigma^2}{2})t}-1}{\beta+L_b+\frac{L_\sigma^2}{2}},\ \ t\geq 0,
\end{align}
and
\begin{align}\label{EQ3}
\nonumber|\Gamma^{\mu_0,x}_t-\Gamma^{\nu_0,y}_t|^2&\leq2\e^{-2\beta t}|x-y|^2+2L^2_b\left|\int_0^t\e^{-\beta(t-s)}\mathbb{W}_{2}(P_s^\ast \mu_0,P_s^\ast\nu_0)\d s\right|^2\\
&\leq 2\e^{-2\beta t}|x-y|^2+2L^2_b\mathbb{W}_2(\mu_0, \nu_0)^2\left|\int_0^t\e^{-\beta(t-s)}\e^{(-\beta+L_b+\frac{L_\sigma^2}{2})s}\d s\right|^2\\
\nonumber&\leq 2\e^{-2\beta t}|x-y|^2+2L^2_b\mathbb{W}_2(\mu_0, \nu_0)^2\e^{-2\beta t}\frac{(\e^{(L_b+\frac{L_\sigma^2}{2})t}-1)^2}{(L_b+\frac{L_\sigma^2}{2})^2},\ \ t\geq 0.
\end{align}
It follows from \eqref{DDV000} that
\begin{align}\label{EQQ}
\nonumber&\mathrm{Ent}(\L_{X^{\mu_0,x}_t}|\L_{X^{\nu_0,y}_t})\\
\nonumber&=\int_{\R}\log\left\{\frac{\d \L_{X^{\mu_0,x}_t}}{\d \L_{X^{\nu_0,y}_t}}(z)\right\} \L_{X^{\mu_0,x}_t}(\d z)\\
&=\log\frac{\sqrt{\Sigma^{\nu_0}_t}}{\sqrt{\Sigma^{\mu_0}_t}} +\int_{\R}\frac{(\Sigma^{\mu_0}_t-\Sigma^{\nu_0}_t)(z-\Gamma^{\mu_0,x}_t)^2+\Sigma^{\mu_0}_t(\Gamma^{\mu_0,x}_t-\Gamma^{\nu_0,y}_t)^2} {2\Sigma^{\mu_0}_t\Sigma^{\nu_0}_t}\L_{X^{\mu_0,x}_t}(\d z)\\
\nonumber &=\log\frac{\sqrt{\Sigma^{\nu_0}_t}}{\sqrt{\Sigma^{\mu_0}_t}} +\frac{(\Sigma^{\mu_0}_t-\Sigma^{\nu_0}_t)} {2\Sigma^{\nu_0}_t}+\frac{(\Gamma^{\mu_0,x}_t-\Gamma^{\nu_0,y}_t)^2} {2\Sigma^{\nu_0}_t},\ \ t>0.
\end{align}
Using Lemma \ref{ine} below for $a=\sqrt{\Sigma^{\nu_0}_t}$ and $b=\sqrt{\Sigma^{\mu_0}_t}$ and submitting \eqref{EQ1}-\eqref{EQ3} into \eqref{EQQ}, we get
\begin{align*}
&\mathrm{Ent}(\L_{X^{\mu_0,x}_t}|\L_{X^{\nu_0,y}_t})\\
&\leq \frac{K+1}{2}\left(\frac{1-\e^{-2\beta t}}{2\beta}\right)^{-2}K^3L_\sigma^2\e^{-4\beta t}\frac{(\e^{(\beta+L_b+\frac{L_\sigma^2}{2})t}-1)^2}{(\beta+L_b+\frac{L_\sigma^2}{2})^2}\mathbb{W}_2(\mu_0, \nu_0)^2\\
&\qquad+\left(\frac{1-\e^{-2\beta t}}{2\beta}\right)^{-1}K\left(\e^{-2\beta t}|x-y|^2+L^2_b\mathbb{W}_2(\mu_0, \nu_0)^2\e^{-2\beta t}\frac{(\e^{(L_b+\frac{L_\sigma^2}{2})t}-1)^2}{(L_b+\frac{L_\sigma^2}{2})^2}\right),\ \ t>0.
\end{align*}
According to \cite[Theorem 1.4.2(2)]{W}, for any $f\in \B_b(\mathbb{R})$ with $f>0$, it holds
\begin{align*}
&\E \log f(X^{\mu_0,x}_t)\leq \log \E f(X^{\nu_0,y}_t)\\
&\leq \frac{K+1}{2}\left(\frac{1-\e^{-2\beta t}}{2\beta}\right)^{-2}K^3L_\sigma^2\e^{-4\beta t}\frac{(\e^{(\beta+L_b+\frac{L_\sigma^2}{2})t}-1)^2}{(\beta+L_b+\frac{L_\sigma^2}{2})^2}\mathbb{W}_2(\mu_0, \nu_0)^2\\
&\qquad+\left(\frac{1-\e^{-2\beta t}}{2\beta}\right)^{-1}K\left(\e^{-2\beta t}|x-y|^2+L^2_b\mathbb{W}_2(\mu_0, \nu_0)^2\e^{-2\beta t}\frac{(\e^{(L_b+\frac{L_\sigma^2}{2})t}-1)^2}{(L_b+\frac{L_\sigma^2}{2})^2}\right),\ \ t>0.
\end{align*}
Taking expectation with respect to any $\pi\in\C(\mu_0,\nu_0)$ on both sides of the above inequality firstly, utilizing \eqref{REP} and Jensen's inequality and then taking infimum in $\pi\in\C(\mu_0,\nu_0)$, we
complete the proof.
\end{proof}
\begin{rem}\label{DDI} When $L_\sigma=L_b=0$, Theorem \ref{log} reduces to the classical log-Harnack inequality with $\Sigma(t)=\frac{2\beta K}{\e^{2\beta t}-1}$, see \cite{W} for more distribution independent models. Moreover, the method in the proof of Theorem \ref{log} is also available for multidimensional distribution dependent Ornstein-Uhlenbeck process, where the diffusion coefficient only depends on the distribution.
\end{rem}
\begin{lem}\label{ine} The following inequality holds
$$-\log \left(\frac{b}{a}\right)+\frac{b^2-a^2}{2a^2}\leq \frac{K+1}{2}\frac{(b-a)^2}{a^2},\ \ \sqrt{\frac{1-\e^{-2\beta t}}{2\beta}}\sqrt{K^{-1}}\leq a, b\leq\sqrt{\frac{1-\e^{-2\beta t}}{2\beta}} \sqrt{K}.$$
\end{lem}
\begin{proof} Let $\frac{b-a}{a}=y$, then $b=a(1+y), K^{-1}-1\leq y\leq  K-1$.  So, it is sufficient to prove
\begin{align}\label{yyy}-\log (1+y)+\frac{y^2+2y}{2}\leq \frac{K+1}{2}y^2,\ \ K^{-1}-1\leq y\leq  K-1.
\end{align}
Define
$$F(y)=-\log (1+y)+\frac{y^2+2y}{2}-\frac{K+1}{2}y^2,\ \ K^{-1}-1\leq y\leq  K-1.$$
It is easy to see that
$$F'(y)=-\frac{1}{1+y}+1+y-(K+1) y=\frac{Ky(K^{-1}-1-y)}{1+y},\ \ K^{-1}-1\leq y\leq  K-1.$$
Since $y\geq K^{-1}-1$, we conclude that $F(y)$ takes maximum value at $y=0$, i.e.
$$F(y)\leq F(0)=0,\ \ K^{-1}-1\leq y\leq  K-1.$$
Therefore, \eqref{yyy} holds and we complete the proof.
\end{proof}
As an application of Theorem \ref{log}, we present the exponential ergodicity of $P_t^\ast$ in relative entropy.
\begin{thm}\label{EXPEN}
Assume that ${\bf(H1)}-{\bf (H2)}$ hold with $\beta>L_b+\frac{L_\sigma^2}{2}$.
Then $P_t^*$  has a unique  invariant probability measure $\mu\in \scr P_2$ with
\begin{align*}&\max(\W_2(P_t^\ast\nu, \mu)^2,\mathrm{Ent}(P_t^\ast\nu|\mu))\\
&\qquad \leq K(t)\e^{-2(\beta-L_b-\frac{L_\sigma^2}{2}) t}\min(\W_2(\nu, \mu)^2,\mathrm{Ent}(\nu|\mu)),\ \ \nu\in\scr P_2,t>0\end{align*}
for some function $K:(0,\infty)\to[0,\infty)$.
\end{thm}
\begin{proof}
When $\beta>L_b+\frac{L_\sigma^2}{2}$, it is standard to derive from \eqref{WET} that $P_t^\ast$ has a unique invariant probability measure $\mu\in\scr P_2$ with
$$\W_2(P_t^\ast\nu, \mu)^2\leq \e^{-2(\beta-L_b-\frac{L_\sigma^2}{2}) t}\W_2(\nu, \mu)^2,$$
see \cite[Proof of Theorem 3.1(2)]{FYW1}.
Consider classical SDE:
\begin{align}\label{Hom}\d X_t=(\gamma-\beta X_t)\d t+b(\mu)\d t+\sigma(\mu)\d W_t.
\end{align}
Since $\beta>0$, it is clear that $\mu$ is the unique invariant probability measure of \eqref{Hom}. Repeating the proof of \cite[(4.2)]{RW}, we can get the log-Sobolev inequality
$$\mu(f^2\log f^2)\leq c\mu(|\nabla f|^2),\ \ f\in C_b^1(\R), \mu(f^2)=1$$
for some constant $c>0$.
According to \cite{BGL}, this implies the Talagrand inequality
$$\W_2(\nu, \mu)^2\leq c\mathrm{Ent}(\nu|\mu), \ \ \nu\in\scr P_2.$$
Combining \cite[Theorem 2.1]{RW} and Theorem \ref{log}, the proof is completed.
\end{proof}
\beg{thebibliography}{99}




\bibitem{BGL} Bobkov, S. G.,  Gentil, I., Ledoux, M., \emph{Hypercontractivity of Hamilton-Jacobi equations,} J. Math. Pures Appl., 80(2001), 669-696.

\bibitem{BH} Bao, J., Huang, X., \emph{Approximations of McKean-Vlasov Stochastic Differential Equations with Irregular Coefficients,} J. Theoret. Probab. (2021).


\bibitem{BBP} Bauer, M., Meyer-Brandis, T., Proske, F.,  \emph{Strong Solutions of Mean-Field Stochastic Differential Equations with irregular drift,} Electron. J. Probab., 23(2018), 35 pp.

\bibitem{C} Cairns, A. J. G., \emph{Interest rate models: an introduction,} Princeton University
Press, 2004.

\bibitem{CJM} Chassagneux, J. F., Jacquler, A., Mihaylov, I.,  \emph{An Explicit Euler Scheme with Strong Rate of Convergence for Financial SDEs with Non-Lipschitz Coefficients,} SIAM J. Financial Math., 7(2016), 993-1021.
\bibitem{CR} Chaudru de Raynal, P. E., \emph{Strong well-posedness of McKean-Vlasov stochastic differential equation with H\"older drift, } Stochastic Process. Appl.,  130(2020), 79-107.

\bibitem{CF} Chaudru de Raynal, P. E., Frikha, N., \emph{Well-posedness for some non-linear diffusion processes and related pde on the Wasserstein space,} arXiv:1811.06904.

\bibitem{CIR} Cox, J. C., Ingersoll, J. E., Ross, S. A.,  \emph{A theory of the term structure of interest rates,} Econometrica, 53(1985), 385-407.


\bibitem{CKLS} Chan, K. C, Karolyi, G. A, Longstaff, F. A., Sanders, A., \emph{An empirical comparison of alternative models of the short-term interest rate,} J. Finance, 47(1992), 1209-1227.

\bibitem{HW} Huang, X., Wang, F.-Y.,  \emph{Distribution dependent SDEs with singular coefficients, } Stochastic Process. Appl., 129(2019), 4747-4770.

\bibitem{HZ} Huang, X., Zhao, F., \emph{Harnack and super Poincar\'{e} inequalities for generalized Cox-Ingersoll-Ross model,} Stoch. Anal. Appl., 38(2020), 730-746.

\bibitem{IW} Ikeda, N., Watanabe, S.,  \emph{Stochastic differential equations and diffusion processes,} 2nd ed. Amsterdam: North Holland, 1989.



\bibitem{MV} Mishura, Yu. S., Veretennikov, A. Yu., \emph{Existence and uniqueness theorems for solutions of McKean-Vlasov stochastic equations,} arXiv:1603.02212.


\bibitem{RW} Ren, P., Wang, F.-Y., \emph{Exponential convergence in entropy and Wasserstein for McKean-Vlasov SDEs,}  Nonlinear Anal., 206(2021), 112259, 20 pp.

\bibitem{RZ} R\"ockner, M., Zhang, X., \emph{Well-posedness of distribution dependent SDEs with singular drifts,} Bernoulli, 27(2021), 1131-1158.
\bibitem{V} Vasicek, O., \emph{An equilibrium characterization of the term structure,} J. Financ. Econ., 5(1977), 177-188.
\bibitem{WMC} Wu, F., Mao, X., Chen, K.,  \emph{The Cox-Ingersoll-Ross model with delay and strong convergence of its Euler-Maruyama approximate solutions,} Appl. Numer. Math., 59(2009), 2641-2658.



\bibitem{W} Wang, F.-Y., \emph{Harnack Inequalities for Stochastic Partial Differential Equations,} Springer, New York, 2013.



\bibitem{FYW1} Wang, F.-Y., \emph{Distribution-dependent SDEs for Landau type equations,} Stochastic Process. Appl., 128(2018), 595-621.


\bibitem{YW} Yang, X., Wang, X.,  \emph{A transformed jump-adapted backward Euler method for jump-extended CIR and CEV models,} Numer. Algorithms, 74(2017), 39-57.


\bibitem{ZZ} Zhang, S.-Q., Zheng, Y.,  \emph{Functional inequalities and the spectral theory for one-dimensional CIR process,} (Chinese) Beijing Shifan Daxue Xuebao, 54(2018), 572-582.
\end{thebibliography}

\end{document}